\documentclass[10pt]{amsart}
\usepackage{amsthm, amsmath,cleveref, amsfonts, amssymb, enumerate, textcmds, tensor, enumitem, mathdots}
\usepackage{pst-node}
\usepackage{tikz-cd} 
\usepackage{graphicx,tikz}
\usepackage{enumerate}
\usepackage{pinlabel}
\usepackage[colorinlistoftodos]{todonotes}
\newcommand{\Z}{\mathbb Z}
\newcommand{\R}{\mathbb R}
\newcommand{\C}{\mathbb C}

\newtheorem{thm}{Theorem}[section]
\newtheorem{lem}[thm]{Lemma}
\newtheorem{prop}[thm]{Proposition}

\newtheorem*{thma}{Theorem A}
\newtheorem*{thmb1}{Theorem B1}
\newtheorem*{thmb2}{Theorem B2}
\newtheorem*{thmc}{Theorem C}
\newtheorem*{thmd}{Theorem D}

\DeclareMathOperator{\QD}{\textrm{QD}}

\theoremstyle{definition}
\newtheorem{defn}[thm]{Definition}
\newtheorem{remark}[thm]{Remark}

\begin{document}
\title[Unstable minimal surfaces in $\mathbb{R}^n$ and in products of hyperbolic surfaces]{Unstable minimal surfaces in $\mathbb{R}^n$ and in products of hyperbolic surfaces}
\author[Vladimir Markovi{\'c}]{Vladimir Markovi{\'c}}
\address{Vladimir Markovi{\'c}: University of Oxford, All Souls College, Oxford, OX1 4AL, UK.} \email{markovic@maths.ox.ac.uk} 
\author[Nathaniel Sagman]{Nathaniel Sagman}
\address{Nathaniel Sagman: Caltech, 1200 E California Blvd, Pasadena, CA, 91125, USA.} \email{nsagman@caltech.edu}
\author[Peter Smillie]{Peter Smillie}
\address{Peter Smillie: Caltech, 1200 E California Blvd, Pasadena, CA, 91125, USA.} \email{smillie@caltech.edu}

\begin{abstract}
We prove that every unstable equivariant minimal surface in $\mathbb{R}^n$ produces a maximal representation of a surface group into $\prod_{i=1}^n\textrm{PSL}(2,\mathbb{R})$ together with an unstable minimal surface in the corresponding product of closed hyperbolic surfaces. To do so, we lift the surface in $\R^n$ to a surface in a product of $\mathbb{R}$-trees, then deform to a surface in a product of closed hyperbolic surfaces. We show that instability in one context implies instability in the other two.
\end{abstract}
\maketitle
\begin{section}{Introduction}

\begin{subsection}{Minimal surfaces in products of hyperbolic surfaces}
Let $\Sigma_g$ denote a closed surface of genus $g\geq 2$ and let $\mathbf{T}_g$ be the Teichm{\"u}ller space of marked complex structures on $\Sigma_g$. Let $(X,d)$ be the hyperbolic plane, an $\mathbb{R}$-tree, or product thereof with an action $\sigma: \pi_1(\Sigma_g) \to \mathrm{Isom}(X,d)$. For every Riemann surface structure $S$ on $\Sigma_g$, with universal cover $\tilde{S}$, and $\sigma$-equivariant Lipschitz map $f:\tilde{S}\to (X,d),$ there is a well-defined notion of Dirichlet energy $\mathcal{E}(S,f)$ (see Section 2 for details). For admissible $\sigma$, there is an essentially unique $\sigma$-equivariant harmonic map $h: \tilde{S} \to (X,d)$, which satisfies $$\mathcal{E}(S,h)=\inf_f \mathcal{E}(S,f).$$ This gives a function $\mathbf{E}_\sigma: \mathbf{T}_g \to \R$, by $\mathbf{E}_\sigma(S)=\mathcal{E}(S,h)$. When $S$ is a critical point of $\mathbf{E}_\sigma$, we say that $h$ is minimal; if $X$ is a manifold and $h$ is an immersion, this is equivalent to $h(S)$ being a minimal surface.

One case of interest is when $\sigma$ is a product of Fuchsian representations into $\textrm{PSL}(2,\mathbb{R})^n$ (also called a maximal representation), in which case each component of the harmonic map is a diffeomorphism, and critical points correspond to genuine minimal surfaces in a product of hyperbolic surfaces. The work of Schoen-Yau \cite{SYm} implies that in this case, $\mathbf{E}_\sigma$ is proper, and therefore admits a global minimum, which is a stable critical point. For $n=2$, Schoen proved that this is the unique critical point of $\mathbf{E}_\sigma$ \cite{Sc}. 

However, the first author proved in \cite{M2} that uniqueness fails when $n\geq 3$, assuming the genus $g$ is large enough. See also the paper \cite{M1}, which provides a strengthening of Schoen's result for $n=2$. The main goal of this paper is to show that unstable equivariant minimal surfaces in $\R^n$ yield unstable minimal surfaces in products of hyperbolic surfaces. In particular, this strengthens the result from \cite{M2}, while providing a simpler and more revealing proof. 
When $n\geq 3$, there are many unstable equivariant minimal surfaces in $\mathbb{R}^n;$ most notably, unstable minimal surfaces in tori, which Meeks \cite{Mee}, Hass-Pitts-Rubenstein \cite{HPR}, and Traizet \cite{Tr} have shown to be abundant, lift to unstable equivariant minimal surfaces in $\mathbb{R}^n$.

We say that a critical point of $\mathbf{E}_\sigma$ is unstable if there exists a $C^2$ path in $\mathbf{T}_g$ starting at the point and at which the second derivative of $\mathbf{E}_\sigma$ along the path is negative.
\begin{thma}
Let $n\geq 3.$ For every genus $g\geq 2$, there exists a maximal representation $\sigma:\pi_1(\Sigma_g)\to \prod_{i=1}^n\textrm{PSL}(2,\mathbb{R})$ such that $\mathbf{E}_\sigma:\mathbf{T}_g\to (0,\infty)$ admits an unstable critical point. In particular, there are at least two minimal surfaces in the product of hyperbolic surfaces determined by $\sigma.$
\end{thma}
Labourie conjectured that for a Hitchin representation into a simple split real Lie group $G$ of non-compact type, there exists a unique equivariant minimal surface in the corresponding symmetric space. Labourie proves existence in general \cite{La1}, and that uniqueness holds when the rank of $G$ is $2$ \cite{La} (see also \cite{CTT}, where Collier-Tholozan-Toulisse prove the analogous statement for maximal representations into Hermitian Lie groups of rank $2$). The conjecture remains open in rank at least $3$, and \cite{M2} suggests that this is the critical case.

The key idea of the proof of Theorem A is to reduce it to finding unstable minimal surfaces in products of $\mathbb{R}$-trees (Theorem B2 below). The unstable minimal surfaces are provided by Theorems C and D. We explain in the forthcoming subsections.
\end{subsection}

\begin{subsection}{Minimal surfaces in products of $\mathbb{R}$-trees} \label{sec: Rtrees}
We give the definitions about harmonic maps to $\mathbb{R}$-trees in Section \ref{2.2}. Throughout the paper, let $S$ be a Riemann surface structure on $\Sigma_g$ and $\QD(S)$ the space of holomorphic quadratic differentials on $S$. The Riemann surface structure $S$ lifts to a Riemann surface structure on the universal cover of $\Sigma_g$, which we denote $\tilde{S}$. Given a non-zero $\phi\in \QD(S)$, there are two natural ways of producing an equivariant harmonic map. First, the leaf space of the vertical singular foliation of the lift $\tilde{\phi}$ to $\tilde{S}$ is an $\R$-tree $(T_\phi,d)$. The action of $\pi_1(\Sigma_g)$ on $\tilde{S}$ descends to an action $\rho: \pi_1(\Sigma_g) \to \mathrm{Isom}(T_\phi,d)$ by isometries. The quotient map $\pi: \tilde{S}\to (T_\phi,d)$ is harmonic and $\rho$-equivariant, with Hopf differential $\phi/4$. 

On the other hand, it is proved independently by Hitchin \cite{Hi}, Wan \cite{Wan}, and Wolf \cite{W2} that there is a unique hyperbolic structure $M_\phi$ on $\Sigma_g$ such that the identity map from $S$ to $M_\phi$ is harmonic with Hopf differential $\phi$. Moreover Wolf shows that as $t \to \infty$, $M_{t\phi}$ converges in a certain sense to the rescaled tree $(T_\phi, 2d)$ (see \cite{W} for the precise statement).

Now let $\phi_1, \ldots, \phi_n$ be $n$ nonzero holomorphic quadratic differentials on the same surface $S$, and let $X$ be the product of the $\mathbb{R}$-trees $(T_{\phi_i},2d_i)$ arising from the construction above. Let $\rho:\pi_1(\Sigma_g)\to \mathrm{Isom}(X)$ be the product of the actions $\rho_i$ on each factor. The energy function $\mathbf{E}_\rho$ on $\mathbf{T}_g$ associated to $\rho$ is then the sum of the energy functions $E_{\rho_i}$ associated to each component. Also for each positive $t > 0$, let $M_i^t$ be the hyperbolic structures associated to $t\phi_i$. We set $\mathbf{E}_{\rho}^t$ to be the energy functional for the product of Fuchsian representations associated to the $M_i^t$.

$S$ is a critical point for $\mathbf{E}_{\rho}$ if and only if it is a critical point for $\mathbf{E}_{\rho}^t$ for every $t>0$. In other words, minimality of the harmonic map into the product of surfaces is equivalent to the minimality of the equivariant harmonic map into the product of $\mathbb{R}$-trees. The condition occurs precisely when $\sum_{i=1}^n \phi_i=0.$ 

Let $n \geq 2$. For $i=1,\dots, n$, let $\phi_i$ be nonzero holomorphic quadratic differentials on the Riemann surface $S$ such that $\sum_{i=1}^n \phi_i=0.$
\begin{thmb1}
$S$ is not a (local) minimum for $\mathbf{E}_\rho$ if and only if there exists $t>0$ such that $S$ is not a (local) minimum for $\mathbf{E}_\rho^t$. In this case, for all $s>t$, $S$ is not a (local) minimum for $\mathbf{E}_\rho^s$.
\end{thmb1}
\begin{remark}
If $n=2$, Schoen's result shows that the only critical point of $\mathbf{E}^t_\rho$ is a minimum, and so by (1), the same is true of $\mathbf{E}_\rho$. This was first proved by Wentworth who showed that, provided existence, the equivariant minimal surface in a product of two $\mathbb{R}$-trees is unique \cite[Theorem 1.6]{We}.
\end{remark}
\begin{remark}
It appears to be unknown whether the energy functional on Teichm{\"u}ller space for harmonic maps to $\mathbb{R}$-trees is $C^2$. It is always $C^1,$ and real analytic near a Riemann surface such that the Hopf differential of the harmonic map has only simple zeros (this is a generic condition) \cite{Ma}.
\end{remark}
Theorem B1 can give critical points of $\mathbf{E}_\rho^t$ that are not minima, but this is not quite strong enough to prove Theorem A, which is about unstable critical points. To that end, we give a notion of instability in products of $\mathbb{R}$-trees that will be suitable for our purposes. Let $S$ be a critical point for $\mathbf{E}_\rho$ with harmonic map $\pi=(\pi_1,\dots,\pi_n)$. Given $C^\infty$ vector fields $V_1,\dots, V_n$ on $S$, let $r\mapsto f_1^r,\dots, f_n^r:S\to S$ be their flows, and construct the map $\pi_r=(\pi_1\circ f_1^r,\dots, \pi_n\circ f_n^r)$. For any $C^\infty$ path of Riemann surfaces $r\mapsto S_r,$ there is a Beltrami form $\mu$ representing a point $T_S\mathbf{T}_g$ that is tangent to our path at $r=0$.

\begin{defn}
We define a quadratic form $\mathbf{L}: T_S\mathbf{T}_g\times H^0(S,TS)^n\to\mathbb{R}$ by $$\mathbf{L}(\mu,V_1,\dots, V_n) = \frac{d^2}{dr^2}|_{r=0}\mathcal{E}(S_r,\pi_r),$$ where $r\mapsto S_r$ is any path tangent to $\mu$ at $r=0$.

The self-maps index of $S$ for $\mathbf{E}_\rho$ is the maximal dimension of $T_S\mathbf{T}_g\times H^0(S,TS)^n$ on which $L$ is negative definitef. If the index is positive, we say that $S$ is unstable.
\end{defn}

We explain that $\mathbf{L}$ is well-defined in Section 2.2. $\mathbf{L}$ is positive semi-definite on $\{0\}\times H^0(S,TS)^n$, and hence if $\mathbf{L}$ is negative definite on a subspace $U\subset T_S\mathbf{T}_g\times H^0(S,TS)^n$, then $U$ projects injectively to $T_S\mathbf{T}_g\times\{0\}$. Moreover, for any variations $r\mapsto S_r$ and  $r\mapsto\pi_r,$ we have $\mathbf{E}_\rho(S_r)\leq \mathcal{E}(S_r,\pi_r),$ and hence if $\mathbf{L}(\mu,V_1,\dots,V_n)<0,$ then $\mathbf{E}_\rho(S_r)<\mathbf{E}_\rho(S)$ for small $r$. See Remark \ref{twoven} below for more motivation for the definition of $\mathbf{L}$.

\begin{thmb2}
The index of $\mathbf{E}_{\rho}^t$ at $S$ is non-decreasing with $t$, and converges to the self-maps index of $S$ for $\mathbf{E}_\rho$ as $t\to\infty.$ Consequently, $S$ is unstable for $\mathbf{E}_\rho$ if and only if it is unstable for $\mathbf{E}_\rho^t$, for $t$ sufficiently large.
\end{thmb2}

Toward the proof of Theorem A, we only need the ``only if" direction of Theorem B2. We include the ``if" direction and Theorem B1 because they show that $\mathbb{R}$-trees are really at the heart of the result. A conjecture in Higher Teichm{\"u}ller theory is that high energy minimal maps into symmetric spaces converge in an appropriate sense to minimal maps into buildings (see \cite{KPPS}). This is the higher rank generalization of \cite{W}. If our results extend to this setting, then this would suggest that any counterexample to the Labourie conjecture would have to come from an unstable minimal map into a building.

\end{subsection}
\begin{subsection}{Equivariant minimal surfaces in $\mathbb{R}^n$}
In order to use Theorem B2 to prove Theorem A, we construct unstable surfaces in products of $\mathbb{R}$-trees. We start by looking in a more familiar place: Euclidean space $\mathbb{R}^n$.

For $i=1,\dots, n$, let $\alpha_i$ be a non-zero holomorphic $1$-form on the Riemann surface $S$. Lifting to $1$-forms $\tilde{\alpha}_i$ on a universal cover $\tilde{S}$ gives the data of a harmonic map to $\mathbb{R}^n$ via integrating the real parts: $$h= (h_1,\dots, h_n), \hspace{1mm} h_i(z) = \textrm{Re}\int_{z_0}^z \tilde{\alpha_i},$$ unique up to translation. The map $h$ intertwines the action of $\pi_1(\Sigma_g)$ on $\tilde{S}$ with some non-trivial homomorphism $\chi: \pi_1(\Sigma_g)\to \mathbb{R}^n$. %The map $h$ is equivariant for the product homomorphism, which we call $\chi$.
The Hopf differential of $h_i$ is the square $\phi_i = \alpha_i^2$, which descends to a holomorphic quadratic differential on $S$, by the equivariance property. $h$ is weakly conformal if and only if $\sum_{i=1}^n \phi_i=0,$ which is equivalent to $h$ being minimal. 

By the construction of Section \ref{sec: Rtrees}, the Hopf differentials $\phi_i$ also define an action $\rho$ of $\pi_1(\Sigma_g)$ on a product $X$ of $\mathbb{R}$-trees and a $\rho$-equivariant minimal map $\pi$. The map $h$ naturally factors through $\pi$. Let $\mathbf{E}_\chi$ and $\mathbf{E}_\rho$ be the corresponding energy functionals on Teichm\"uller space. In the end we prove the following near-equivalence.

\begin{thmc}
For $n \geq 2$ and $i=1,\dots, n$, let $\alpha_i$ be nonzero holomorphic $1$-forms on $S$ such that $\sum_{i=1}^n \alpha_i^2=0,$. Let $\rho$, and $\chi$ be as above.
\begin{enumerate}
    \item If $S$ is not a (local) minimum for $\mathbf{E}_\rho$, then it is not a (local) minimum for $\mathbf{E}_\chi$.
    \item The index of $\mathbf{E}_{\chi}$ at $S$ is equal to the self-maps index of $\mathbf{E}_\rho$ for $S$. In particular, if $S$ is unstable for $\mathbf{E}_{\chi}$, then $S$ is unstable for $\mathbf{E}_\rho.$ 
\end{enumerate}
\end{thmc}
\begin{remark}
As in Theorem B, the statement is not so interesting when $n=2$ since every critical point is a stable minimum.
\end{remark}

\begin{remark} \label{rmk: area/energy}
Instability for $\mathbf{E}_\chi$ at $S$ is equivalent to instability for the (equivariant) area functional on the space of all equivariant maps. The second variations for both functionals have the same index (see \cite[Theorem 3.4]{Ej}).
\end{remark}

The final ingredient needed to prove Theorem A is an example of an unstable equivariant minimal surface in $\mathbb{R}^n$. Fortunately, these aren't so hard to find: when $n=3,$ every non-planar equivariant minimal surface in $\mathbb{R}^3$ is unstable, since a constant section of the normal bundle is destabilizing. Consequently, for any three $1$-forms on $S$ whose squares sum to zero, as long as they span a 2-dimensional space, $S$ will be an unstable point of $\mathbf{E}_\chi$ (we explain the details in Section \ref{5.3}).

The most natural example is the lift of a minimal surface in a flat $3$-torus; there a many classical examples, such as the Schwarz P-surface of genus $3$ (see \cite{Mee}). In fact, for every $g\geq 3$, every flat $3$-torus contains infinitely many distinct unstable minimal surfaces of genus $g$ in the same homotopy class (see \cite{Mee}, \cite{HPR}, and \cite{Tr}). 

By inclusion, this gives examples for every $n\geq 3$, as long as $g \geq 3$. We can also perturb these examples to give even more. Unfortunately, if $g=2$, then the only triples of 1-forms whose squares sum to zero are scalar multiples of one another, so we cannot use Theorem C to prove Theorem A in the case $g=2$.
\end{subsection}

\begin{subsection}{A generalization of Theorem C, and the case $g=2$}
In the general setting where we start with $n$ quadratic differentials summing to zero that are not necessarily squares of abelian differentials, we may have to lift to a branched covering of $S$ in order to get a harmonic map to $\R^n$. Replacing $\pi_1(\Sigma)$ with the Deck group of this branched covering, and $\chi$ with the corresponding representation of this group, we prove that an analog of statement (2) from Theorem C still holds.

When $n=3$, the minimal surface in $\R^3$ arising from a branched cover is no longer automatically equivariantly unstable for its energy functional, because the normal bundle is not necessarily equivariantly trivial with respect to the action of the Deck group of the branched covering. Instead, we find a condition on the quadratic differentials that guarantees that the bundle is equivariantly trivial, which then yields the following theorem.

 \begin{thmd}
 Let $\phi_1,\phi_2,\phi_3$ be holomorphic quadratic differentials on $S$ that are not colinear and such that $\phi_1\phi_2\phi_3$ is the square of a cubic differential. Then the corresponding equivariant minimal surface in the product of $\mathbb{R}$-trees is unstable via self-maps for its energy functional.
 \end{thmd}
 
We show that the moduli space of solutions to this problem has dimension at least $3g-3$ for every genus (Proposition \ref{3g-3}). In particular, this gives us unstable minimal surfaces in genus 2, and hence allows us to complete the proof of Theorem A.

\end{subsection}

\begin{subsection}{Acknowledgments}
We would like to thank Phillip Engel and Mike Wolf for helpful discussions. Vladimir Markovi{\'c} is supported by the Simons Investigator Award 409745 from the Simons
Foundation.
\end{subsection}

\end{section}

\begin{section}{Preliminaries}\label{2}
\begin{subsection}{Harmonic maps to manifolds}
Let $\nu$ be a smooth metric on $S$ compatible with the complex structure. Let $(M,\sigma)$ be a closed Riemannian manifold, and $h:S\to M$ a $C^2$ map. The energy density is the function
\begin{equation}\label{en}
e(h) = \frac{1}{2}\textrm{trace}_\nu h^*\sigma,
\end{equation}
and the total energy is
\begin{equation}\label{tot}
    \mathcal{E}(S,h) = \int_S e(h)dA,
\end{equation}
where $dA$ is the area form of $\nu$. We comment here that the energy density $2$-form $e(h)dA$ does not depend on the choice of compatible metric $\nu$, but only on the complex structure. $h$ is said to be harmonic if it is a critical point for the energy $h\mapsto \mathcal{E}(S,h).$ The Hopf differential is the quadratic differential on $S$ defined in a local holomorphic coordinate $z$ by 
\begin{equation} \label{hopf}
\phi(h)(z) = h^*\sigma\Big (\frac{\partial}{\partial z},\frac{\partial}{\partial z}\Big )(z)dz^2,
\end{equation}
and it is holomorphic provided that $h$ is harmonic. 
We also consider maps from $\tilde{S}\to (M,\sigma)$ that are equivariant with respect to a representation $\rho: \tilde{S}\to \textrm{Isom}(M,\nu).$ Since $\rho$ is acting by isometries, the energy density is invariant under the $\pi_1(\Sigma_g)$ action on $\tilde{S}$ by deck transformations, and hence descends to a function on $S$. In this way, we can define a total energy exactly as in (\ref{tot}), and discuss harmonic maps. Similarly, the Hopf differential descends to a holomorphic quadratic differential on $S$.

We give special attention to surfaces and the real line. When the target $(M,\sigma)$ is a surface with conformal metric $\sigma$, then in holomorphic coordinates $z$ on $S$ and $w$ on $M$, we write $\nu=\nu(z)|dz|^2$, $\sigma=\sigma(w)|dw|^2$, and $h$ as a complex-valued function $h(z)$. The energy density takes the form $$e(h)(z) = \frac{\sigma(h(z))}{\nu(z)}(|h_z|^2 + |h_{\overline{z}}|^2)(z), \hspace{1mm} \phi(h)(z) = \sigma(h(z))h_z\overline{(h_{\overline{z}})}dz^2.$$ Considering equivariant maps to the real line, again in local coordinates, $$e(h)(z) = 2\nu(z)^{-1}|h_z|^2, \hspace{1mm} \phi(h)=h_z^2dz^2.$$

For a harmonic map to a product space, the definitions (\ref{en}) and (\ref{hopf}) shows that the energy density and the Hopf differential are the sum of the energy densities and the Hopf differentials respectively of the component maps. So for a mapping $h=(h_1,\dots, h_n)$ into a product of Riemann surfaces, or an equivariant mapping into $\mathbb{R}^n,$ the Hopf differential is the sum
\begin{equation}\label{hsum}
    \phi(h) = \sum_{i=1}^n \phi(h_i).
\end{equation}
$h$ is minimal if $\phi\equiv 0.$ 

If $(M,\sigma)$ is a negatively curved surface, it is well-known that there is a unique harmonic map $h:S\to (M,\sigma)$ in the homotopy class of the identity (see \cite{ES} for existence, and \cite[Theorem H]{Har} for uniqueness). If we work on a different Riemann surface structure $S'$ on $\Sigma_g$, we get a harmonic map from $S'\to (M,\sigma)$ in the class of the identity, and the total energy depends only on the class of $S'$ in Teichm{\"u}ller space. Thus, we get a functional $\mathbf{E}:\mathbf{T}_g\to (0,\infty)$, where $\mathbf{E}(S')$ is the total energy of the harmonic map from $S'\to (M,\sigma).$ For a map into a product of surfaces, the energy functional is the sum of the energy functionals of the component mappings.
\end{subsection}

\begin{subsection}{Harmonic maps to $\mathbb{R}$-trees}\label{2.2}
\begin{defn}
An $\mathbb{R}$-tree is a length space $(T,d)$ such that any two points are connected by a unique arc, and every arc is a geodesic, isometric to a segment in $\mathbb{R}$.
\end{defn}
 A point $x\in T$ is a vertex if the complement $T\backslash\{x\}$ has greater than two components. Otherwise it is said to lie on an edge.

Let $S$ be a closed Riemann surface of genus $g\geq 2$. The horizontal (resp. vertical) foliation of a holomorphic quadratic differential $\phi$ on $S$ is the singular foliation whose non-singular leaves are the integral curves of the line field on $S\backslash \phi^{-1}(0)$ on which $\phi$ is a negative real number. The singularities are standard prongs at the zeros. Both foliations come equipped with transverse measures $|\textrm{Im}\sqrt{\phi}|$ and $|\textrm{Re}\sqrt{\phi}|$ respectively (see \cite[Expos{\'e} 5]{Thbook} for precise definitions).

In this paper, we work with the vertical foliation, unless specified otherwise. Lifting to a singular measured foliation on a universal cover $\tilde{S}$, we define an equivalence relation under which two points $x,y\in\tilde{S}$ are equivalent if they lie on the same leaf. The quotient space is denoted $T$, and we can push the transverse measure down via the projection $\pi: \tilde{S}\to T$ to form a distance function $d$ such that $(T,d)$ is an $\mathbb{R}$-tree, with an induced action $\rho:\pi_1(S)\to \textrm{Isom}(T,d).$

According to Korevaar-Schoen, for Lipschitz maps $f$ from $\tilde{S}$ to complete and non-positively curved (NPC) length spaces such as $(T,d)$, there is a well-defined $L^1$ directional energy tensor $g_{ij}=g_{ij}(f)$ that generalizes the pullback metric (see \cite[Theorem 2.3.2]{KS}). In this way, one can define a measurable energy density function by 
\begin{equation}\label{mease}
    e(f)=\frac{1}{2}\textrm{trace}_\nu g_{ij}(f).
\end{equation}
For an equivariant Lipschitz map $h$, the energy density $e(h)$ is invariant under the group, and we define a total energy as in the smooth setting by $$\mathcal{E}(S,h) = \int_S e(h)dA.$$ 
\begin{defn}
We say that a $\rho$-equivariant map $h: \tilde{S}\to (T,d)$ is harmonic if, among other $\rho$-equivariant maps, it is a critical point for the energy $h\mapsto \mathcal{E}(S,h).$ 
\end{defn}
For the projection map $\pi$, we can describe the energy density explicitly. At a point on $p\in\tilde{S}$ on which $\phi(p)\neq 0$, the map locally isometrically factors through a segment in $\mathbb{R}$. In a small enough neighbourhood around that point, $e(\pi)$ is equal to the energy density of the locally defined map to $\mathbb{R}$, which is computed as usual via the formula (\ref{en}). From this, we see that the energy density has a continuous representative that is equal to $\nu^{-1}|\phi|/2$ everywhere.

The Hopf differential is well-defined for maps $f$ from $\tilde{S}$ to NPC spaces as above: in local coordinates, it is given by 
\begin{equation}\label{mhopf}
\frac{1}{4}(g_{11}(f)(z)-g_{22}(f)(z)-2ig_{12}(f)(z))dz^2.
\end{equation}
The projection map $\pi: \tilde{S}\to (T,d)$ is $\rho$-equivariant and harmonic, with Hopf differential $\phi/4$. Instead of the equation (\ref{mhopf}), one can also see this by using the local isometric factoring. As in the case of maps to surfaces, a harmonic mapping into a product of trees is called minimal if the Hopf differential--which splits as a sum as in (\ref{hsum})--vanishes.

Given $(T,d)$ as above, we always rescale the metric to $2d$, which makes it so that the Hopf differential of $\pi:\tilde{S}\to (T,2d)$ is $\phi$. For any other Riemann surface $S'$ representing a point in $\mathbf{T}_g$, there is a unique $\rho$-equivariant harmonic map from $\tilde{S}'\to (T,2d)$ (see \cite{Wf}).  Again like the surface case, the representation $\rho$ then defines an energy functional on Teichm{\"u}ller space. The same holds for products of $\mathbb{R}$-trees with admissible actions.

Let's now address the quadratic form $\mathbf{L}: T_S\mathbf{T}_g\times H^0(S,TS)^n\to\mathbb{R}$. Given a $C^\infty$ path of Riemann surfaces and a flow $r\mapsto f_r,$ we consider
\begin{equation}\label{2vvar}
    r\mapsto \mathcal{E}(S_r,\pi\circ \tilde{f_r}).
\end{equation}
In \cite[Section 5]{Moo}, Moore computes the derivative of the two-variable energy for maps from a surface to a Riemannian manifold. Using the characterization (\ref{mease}), one can word-for-word redo that computation, but with the measurable density $e(\pi\circ \tilde{f}_r)$, to see that (\ref{2vvar}) is $C^\infty$ in $r$. Working with a minimal map into a product of trees and $n$ vector fields, one does the computation $n$ times to get $\mathbf{L}$. We note that it only depends on the tangent vectors and not the specific path of Riemann surfaces and flow of maps, because a minimal map is a critical point for the two-variable energy.
\end{subsection}

\end{section}

\begin{section}{Minimal surfaces in products of $\mathbb{R}$-trees}
We first prove Theorem B2, and then Theorem B1.
\begin{subsection}{The Reich-Strebel energy formula}
Reich-Strebel computed a formula for the difference of energies of quasiconformal maps (equation 1.1 in \cite{RS}). Let $h:S\to M$ and $f:S\to S'$ be quasiconformal maps between Riemann surfaces, with a conformal metric on $M$. Let $\mu$ be the Beltrami form of $f$, and $\phi$ the Hopf differential of $h$, which need not be holomorphic. Then,
\begin{equation}\label{RSorig}
    \mathcal{E}(S',h\circ f^{-1}) -\mathcal{E}(S,h) =  -4\textrm{Re} \int_S \phi\cdot \frac{ \mu}{1-|\mu|^2} + 2\int_S e(h)\cdot \frac{|\mu|^2}{1-|\mu|^2}dA.
\end{equation}
The computation goes through just the same when $h$ maps $\tilde{S}$ equivariantly into an $\mathbb{R}$-tree. For a map $h$ to a tree $(T,d)$ with Hopf differential $\psi$, the energy density satisfies $e(h)=2\nu^{-1}|\psi|.$ We obtain the Proposition below.
\begin{prop}\label{RStree}
Let $h: \tilde{S}\to (T,d)$ be an equivariant harmonic map to an $\mathbb{R}$-tree with Hopf differential $\psi$, and $f:S'\to S$ a quasiconformal map. Then the following formula holds: 
\begin{equation}\label{RSfor}
    \mathcal{E}(S',h\circ \tilde{f}^{-1}) -\mathcal{E}(S,h) =  -4\textrm{Re} \int_S \psi\cdot \frac{ \mu}{1-|\mu|^2} + 4\int_S |\psi|\cdot \frac{|\mu|^2}{1-|\mu|^2}.
\end{equation}
\end{prop}
In the formula above, $\tilde{f}$ is the lift to $\tilde{S}$ and $\mu$ is the Beltrami form.
\end{subsection}

\begin{subsection}{Proof of Theorem B2}
 Let $\phi\in \QD(S) - \{0\}$. For $t>0,$ let $M_t$ be the hyperbolic structure with hyperbolic metric $\sigma_t$ such that the identity map $h_t:S\to M_t$ is harmonic with Hopf differential $t\phi$, let $\mathcal{E}_\rho^t$ be the two-variable energy functional for $M_t$, and $E_\rho^t$ the corresponding energy functional on Teichm{\"u}ller space. Similarly, let $\mathcal{E}_\rho$ and $E_\rho$ be the energies for the $\mathbb{R}$-tree $(T,2d)$ determined by $\phi$ (with a rescaled metric). 

The main step in the proof of Theorem B2 is Lemma \ref{2conv}. If we rescale $\sigma_t$ by $t^{-1}$, then for any Riemann surface structure $S'$ and $C^2$ map $f:S'\to M_t,$ the energy with respect to the target metric $t^{-1}\sigma_t$ is $t^{-1}\mathcal{E}_\rho^t(S',f)$. Let $r\mapsto S_r$ be a path of Riemann surfaces and $r\mapsto f_r$ a flow starting at the identity map.  Lemma \ref{2conv} shows that the second derivative in $r$ of the energy of $h_t\circ f_r$ on $S_r$ with respect to the target metric $t^{-1}\sigma_t$ converges as $t\to\infty$ to the second derivative of the energy of $\pi\circ f_r:\tilde{S}_r\to (T,2d).$
\begin{lem}\label{2conv}
For $s>t$, $$\frac{1}{t}\frac{d^2}{dr^2}|_{r=0}\mathcal{E}_\rho^t(S_r,h_t\circ f_r)>\frac{1}{s}\frac{d^2}{dr^2}|_{r=0}\mathcal{E}_\rho^s(S_r,h_s\circ f_r)>\frac{d^2}{dr^2}|_{r=0}\mathcal{E}_\rho(S_r,\pi\circ \tilde{f}_r),$$ and $$\lim_{t\to\infty}\frac{1}{t}\frac{d^2}{dr^2}|_{r=0}\mathcal{E}_\rho^t(S_r,h_t\circ f_r)=\frac{d^2}{dr^2}|_{r=0}\mathcal{E}_\rho(S_r,\pi\circ \tilde{f}_r).$$
\end{lem}
Toward the proof, we first record the lemma below about the growth of the energy density.
\begin{lem}\label{growth}
Let $e(h_t)$ be the energy density of $h_t$ with respect to the target metric $\sigma_t$. Then for $s\geq t$, 
\begin{equation}\label{mono}
    \frac{e(h_{t})}{t}\geq \frac{e(h_s)}{s},
\end{equation}
and the inequality is strict away from the zeros of $\phi$. Moreover, 
\begin{equation}\label{holim}
    \lim_{t\to\infty} \frac{e(h_{t})}{t}=2\nu^{-1}|\phi|.
\end{equation}
\end{lem}
\begin{proof}
Let $\mu_t$ and $\mu_s$ be the Beltrami forms of $h_t$ and $h_s$ respectively. It is proved in \cite[Proposition 4.3]{W2} that away from the zeros of $\phi$ (at which $|\mu_t|=0$ for every $t$), $|\mu_t|$ monotonically increases to $1$ as $t\to \infty.$ A simple computation gives $$e(h) = \frac{2t|\phi|}{\nu}\cosh \log |\mu_t|^{-1},$$ and likewise for $s.$ Therefore, (\ref{mono}) is equivalent to the inequality
\begin{equation}\label{cosh}
    \cosh \log |\mu_{t_0}|^{-1}\geq \cosh \log |\mu_{t}|^{-1}.
\end{equation}
 Since $|\mu_t|<1$ everywhere, the inequality (\ref{cosh}) follows. Using the limiting behaviour of $|\mu_t|,$ we take $t\to\infty$ to obtain (\ref{holim}).
\end{proof}

\begin{proof}[Proof of Lemma \ref{2conv}]
Let $(\mu,V)\in T_S\mathbf{T}_g\times H^0(S,TS)$, and let $r\mapsto f_r$ be the flow of $V$.  Let $\mu_r$ be the Beltrami form of $f_r^{-1}$, and $\alpha$ the $C^\infty$ $(1,-1)$-form and $\beta$ the $C^\infty$ function on $S$ described by $$\alpha(z) = \frac{d^2}{dr^2}\bigg|_{r=0} \frac{ \mu_r(z)}{1-|\mu_r(z)|^2}, \hspace{1mm} \beta(z)= \frac{d^2}{dr^2}\bigg|_{r=0} \frac{|\mu_r(z)|^2}{1-|\mu_r(z)|^2}.$$
 We use the Reich-Strebel fomula (\ref{RSorig}). For each $t>0$,
\begin{align*}
\frac{1}{t}\frac{d^2}{dr^2}\bigg|_{r=0}\mathcal{E}_\rho^t(S_r,h^t\circ f_r) &= \frac{1}{t}\frac{d^2}{dr^2}\bigg|_{r=0} \Big (\mathcal{E}_\rho^t(S_r,h^t\circ (f_r^{-1})^{-1})- \mathcal{E}_\rho^t(S,h^t)\Big ) \\
    &=  \frac{d^2}{dr^2}\bigg|_{r=0} \Big (-4\textrm{Re} \int_S \phi\cdot \frac{ \mu_r}{1-|\mu_r|^2} + 2\int_S \frac{e(h^t)}{t}\cdot \frac{|\mu_r|^2}{1-|\mu_r|^2}dA\Big ) \\
    &=-4\textrm{Re} \int_S \phi\cdot\alpha + 2\int_S \frac{e(h^t)}{t}\cdot \beta dA.
\end{align*}
On the other hand, by the same computation, but using (\ref{RSfor}),
$$\frac{d^2}{dr^2}\bigg|_{r=0}\mathcal{E}_\rho(S_r,\pi\circ \tilde{f}_r) =-4\textrm{Re} \int_S \phi\cdot\alpha + 4\int_S |\phi(h)|\cdot \beta.$$
By Lemma \ref{growth}, for $s>t$, $$\int_S \frac{e(h^t)}{t}\cdot \beta dA> \int_S \frac{e(h^s)}{s}\cdot \beta dA,$$ so that $\frac{L_t}{t}>\frac{L_s}{s}$. By Lemma \ref{growth} again and the dominated convergence theorem, $$\int_S \frac{e(h^t)}{t}\cdot \beta dA\to 2\int |\phi| \cdot \beta $$ in a strictly decreasing fashion as $t\to\infty$. The result follows.

\end{proof}
Moving onto the proof of Theorem B2, we resume the notation from the introduction: for $i=1,\dots, n,$ $\phi_i$ are nonzero holomorphic quadratic differentials on $S$ summing to $0$. The product of $\mathbb{R}$-trees $(T_i,2d_i)$, which we denote by $X$, comes equipped with the action $\rho=(\rho_1\times\dots\times \rho_n).$ For each positive $t>0$, $M_i^t$ is the hyperbolic structure such that the identity map $h_i^t: S\to M_i^t$ is harmonic and has Hopf differential $t\phi_i$. The energy functionals are denoted $\mathcal{E}_\rho$ and $\mathbf{E}_{\rho}$ for the trees and $\mathcal{E}_\rho^t$ and $\mathbf{E}_{\rho}^t$ for the surfaces.

We define $\mathbf{L}_t: T_s\mathbf{T}_g\times H^0(S,TS)^n\to\R$ for the harmonic map $h_t=(h_1^t,\dots, h_n^t)$ in the same way as $\mathbf{L}$: if $r\mapsto S_r$ is a path of Riemann surfaces tangent to a Beltrami form $\mu$ at $r=0$, and $V_1,\dots, V_n$ are vector fields giving rise to flows $\mapsto f_1^r,\dots, f_n^r,$ then we set  $$\mathbf{L}_t(\mu,V_1,\dots, V_n) = \frac{d^2}{dr^2}|_{r=0}\mathcal{E}_\rho^t(S_r,h_r^t),$$ where $h_t^r=(h_1^t\circ f_1^r,\dots, h_n^t\circ f_n^r).$ 
\begin{remark}\label{twoven}
The two-variable energy for $\prod_{i=1}^n M_i^t$ is defined on $\mathbf{T}_g\times\prod_{i=1}^n\textrm{Maps}(S,M_i^t)$. Since any small perturbation of the identity map is a diffeomorphism, the space on which $\mathbf{L}_t$ acts is canonically isomorphic to the tangent space of $\mathbf{T}_g\times\prod_{i=1}^n\textrm{Maps}(S,M_i^t)$ at $S\times \prod_{i=1}^n \textrm{id}$. 
\end{remark}
\begin{remark}
Since $S$ is a critical point for the two-variables energies $\mathcal{E}_\rho^t$ and $\mathcal{E}_\rho$, the second order derivatives $\mathbf{L}_t$ and $\mathbf{L}$ depend only on the first order data $\mu, V_1,\dots, V_n$ (this is not true for the second variations of each component harmonic map).
\end{remark}

\begin{prop}\label{nconv}
For $s>t$, $$\frac{\mathbf{L}_t}{t}>\frac{\mathbf{L}_s}{s}>\mathbf{L},$$ and $\lim_{t\to\infty}\frac{\mathbf{L}_t}{t}=\mathbf{L}.$
\end{prop}
\begin{proof}
We invoke Lemma \ref{2conv} $n$ times.
\end{proof}

\begin{lem}\label{ltet}
The index of $\mathbf{L}_t$ is equal to the index of $\mathbf{E}_\rho^t.$
\end{lem}
\begin{proof}
Let $r\mapsto S_r$ be a path of Riemann surfaces, tangent to the Beltrami form $\mu$ at $r=0$, and suppose there exists $(V_1,\dots,V_n)\in H^0(S,TS)^n$ such that $\mathbf{L}_t(\mu,V_1,\dots, V_n)<0$. For each fixed $t>0,$ the maps $r\mapsto \mathbf{E}_\rho^t(S_r)$ and $r\mapsto\mathcal{E}_\rho^t(S_r,h_t^r)$ have zero first variation at $r=0$. By the minimizing property for harmonic maps, $\mathbf{E}_\rho^t(S_r)\leq \mathcal{E}_\rho^t(S_r,h_t^r)$ for every $r,$ and it follows that $$\frac{1}{t}\frac{d^2}{dr^2}|_{r=0} \mathbf{E}_\rho^t(S_r)\leq\frac{1}{t}\frac{d^2}{dr^2}|_{r=0}\mathcal{E}_\rho^t(S_r,h_t^r)=\frac{1}{t}\mathbf{L}_t(\mu,V_1,\dots, V_n)<0.$$ So, the index of $\mathbf{E}_\rho^t$ is at least that of $\mathbf{L}_t.$

For the other direction, assume $r\mapsto S_r$ lowers $\mathbf{E}_\rho^t$ to second order, and for each $r>0,$ let $k=(k_1^r,\dots, k_n^r): S_r\to \prod_{i=1}^n M_i^t$ be the harmonic map in the class of the identity. All $h_i^t$'s and $k_i^r$'s are orientation-preserving diffeomorphisms. Set $f_i^r=(h_i^t)^{-1}\circ k_i^r$ and let $V_i$ be the infinitesimal generator of the flow $r\mapsto f_i^r$. Then $$\frac{1}{t}\mathbf{L}_t(\mu,V_1,\dots, V_n)= \frac{1}{t}\frac{d^2}{dr^2}|_{r=0}\mathbf{E}_\rho^t(S_r)<0,$$ which gives the result.
\end{proof}
We now deduce Theorem B2.
\begin{proof}[Proof of Theorem B2]
Proposition \ref{nconv} implies that the index of $\mathbf{L}_t$ is non-decreasing with $t$, and converges to the self-maps index of $S$ for $\mathbf{E}_\rho$. We then apply Lemma \ref{ltet} to obtain the same statement for the index of $\mathbf{E}_\rho^t$ at $S$.
\end{proof}

\end{subsection}
\begin{subsection}{Proof of Theorem B1}
The proof of Theorem B1 is similar to that of Theorem B2, so we don't go through every detail. The main difference is that we replace Lemma \ref{2conv} with Lemma \ref{conv} below.

As above, let $M_t$ be the hyperbolic structure on $\Sigma_g$ with hyperbolic metric $\sigma_t$ such that the identity map has Hopf differential $t\phi$, with energy functional $E_\rho^t$, and let $E_\rho$ be the energy functional for the $\mathbb{R}$-tree $(T,2d)$ for $\phi$. 
\begin{lem}\label{conv}
For all Riemann surfaces $S'$, 
\begin{equation*}
    \lim_{t\to \infty} \frac{E_\rho^t(S')}{t} =E_\rho(S').
\end{equation*}
\end{lem}
In order to prove the lemma, we recall some facts about the Thurston compactification of Teichm{\"u}ller space. Let $\mathcal{S}$ be the set of homotopically non-trivial simple closed curves on $\Sigma_g$ and $\mathbb{R}^{\mathcal{S}}$ the product space with the weak topology. There is an embedding $$\ell:\mathbf{T}_g\times \mathbb{R}^+\to \mathbb{R}^{\mathcal{S}}$$ that associates the data of a hyperbolic metric $\sigma$ and $s\in\mathbb{R}^+$ to the set of lengths of geodesic representatives of curves in $\mathcal{S}$ with respect to the metric $s\sigma$. Every singular measured foliation $(\mathcal{F},\mu)$ on $S$ also defines a point in $\mathbb{R}^{\mathcal{S}},$ by taking $\mu$-transverse measures of simple closed curves. Furthermore, there is an injective map $$\beta: \textrm{QD}(S)\to \mathbb{R}^{\mathcal{S}}$$ that takes a quadratic differential to its vertical foliation, and then to $\mathbb{R}^{\mathcal{S}}.$ Note that both $\ell$ and $\beta$ are homogeneous with exponent $\tfrac{1}{2}$.

According to Thurston and Hubbard-Masur (see \cite{Thbook} and \cite{HMt}), both $\ell$ and $\beta$ are homeomorphisms onto their images, and $\ell(\mathbf{T}_g \times \R^+) \sqcup \beta(QD(S))$ is homeomorphic to a cone over a closed ball, which we call $\mathcal{C}$ (the cone over the Thurston compactification of Teichmuller space). The following result can be gleaned from the results of \cite{W2}.

\begin{thm} \label{Cont} For any Riemann surface $S$, let $E_S: \mathbf{T}_g \times \R^+ \sqcup QD(S) \to \R^+$ be the function that associates to each point in $\mathbf{T}_g \times \R^+$ the energy of the unique harmonic map isotopic to the identity from $S$, and to each point of $QD(S)$ the energy of the unique equivariant harmonic map to the corresponding $\R$-tree. Then $E_S$ is continuous with respect to the topology on $\mathcal{C}$.
\end{thm}

We now explain how to deduce this theorem from the paper \cite{W2}. The first ingredient is a de-projectivized version of Lemma 4.7 of that paper, whose proof is identical to the proof of the lemma in the paper.

\begin{lem} \label{Wo}
Suppose $(\lambda_n)_{n=1}^\infty\subset \mathbf{T}_g$ leaves all compact subsets of the Teichm{\"u}ller space, and let $\phi_n$ be the Hopf differential of the harmonic map from $S$ to $(S,\lambda_n)$. Suppose $(a_n)_{n=1}^\infty\subset (\mathbb{R}^+)^{\mathcal{S}}$ is a chosen sequence. Then $\ell(\lambda_n,a_n)$ converges in $\mathbb{R}^{\mathcal{S}}$ if and only if $\beta(a_n\phi_n)$ does, and in the case of convergence, the two sequences have the same limit.
\end{lem}

The second ingredient is the following computation (in which each term is linear in the scalars $a_n$, so the factors of $a_n$ are superfluous).

\begin{lem}[Lemma 3.2 in \cite{W2}] \label{E estimate} In the notation of the previous lemma,
\[
||a_n\phi_n||_{L^1(S)}\leq E_S(a_n \lambda_n) \leq ||a_n\phi_n||_{L^1(S)} + a_n|\chi(\Sigma_g)|.
\]
\end{lem}

\begin{proof}[Proof of Theorem \ref{Cont}] For brevity, write $E = E_S$. First, $E$ is continuous on $\mathbf{T}_g \times \R^+$ and $E(\phi) = ||\phi||_{L^1(S)}$, which is certainly continuous on $QD(S)$. To show that $E$ is continuous on all of $\mathcal{C}$, we just need to show that if $\ell(\lambda_n, a_n) \to \beta(\phi)$, then $E(a_n \lambda_n) \to E(\phi)$. By Lemma \ref{Wo}, $\beta(a_n\phi_n) \to \beta(\phi)$ (where as above $\phi_n$ is the Hopf differential of the harmonic map to $\lambda_n$), and since $\beta$ is a homeomorphism onto its image, $a_n\phi_n \to \phi$ as well, so $E(a_n\phi_n) \to E(\phi)$. Finally, since $a_n$ must tend to zero in order for the sequence $a_n \lambda_n$ to converge in $\mathbb{R}^{\mathcal{S}}$, Lemma \ref{E estimate} shows that $E(a_n \phi_n)$ and $E(a_n \lambda_n)$ have the same limit.
\end{proof}
Now the proof of Lemma \ref{conv} is easy.
\begin{proof}[Proof of Lemma \ref{conv}] By definition, $E^t_\rho(S') = E_{S'}(\sigma_t)$, and $E_\rho(S') = E_{S'}(\phi)$, so by the continuity of $E_{S'}$ and its homogeneity, we just need to show that $\ell(\sigma_t/t) \to \beta(\phi)$ in $\mathcal{C}$. To prove this, we use Lemma \ref{Wo} applied to the surface $S$. Indeed, the Hopf differential of the harmonic map from $S$ to $\sigma_t/t$ is $\phi$ by construction, and since the constant sequence at $\phi$ trivially converges to $\phi$, Lemma \ref{Wo} implies that $\ell(\sigma_t/t)$ does as well.
\end{proof}

Preparations aside, we prove Theorem B1. We return to all of the notation from the introduction and the proof of Theorem B2. We don't recall it in full, but just record here that the energy functionals are $\mathbf{E}_{\rho}$ for the product of $\mathbb{R}$-trees and $\mathbf{E}_{\rho}^t$ for the product of surfaces. The proof is quite similar to that of Theorem B2, so we leave the details of the computations to the reader.
\begin{proof}[Proof of Theorem B1]
Beginning with a Riemann surface $S'$ such that $\mathbf{E}_\rho(S')<\mathbf{E}_\rho(S)$, applying Lemma \ref{conv} $n$ times yields that $\mathbf{E}_\rho^t(S')<\mathbf{E}_\rho^t(S)$ for sufficiently large $t$.

Conversely, suppose that there exists $t>0$ such that $\mathbf{E}_\rho^t(S')<\mathbf{E}_\rho^t(S)$, and let $k=(k_1^t,\dots, k_n^t): S'\to \prod_{i=1}^n M_i^t$ be the $n$-tuple of harmonic diffeomorphisms with lower energy. Let $h_i^t$ be the $i^{th}$ component of the harmonic map $h_t$, and set $f_i^t=(h_i^t)^{-1}\circ k_i^t$. Arguing similarly to the proof of Lemma \ref{2conv}, Reich-Strebel formulas (\ref{RSorig}) and (\ref{RSfor}) and the monotonicity on the level of energy densities from Lemma \ref{growth} show that for $s>t$,
\begin{align*}
    \frac{\mathbf{E}_\rho^t(S')- \mathbf{E}_\rho^t(S)}{t}&=\frac{\sum_{i=1}^n\mathcal{E}_\rho^t(S',h_i^t\circ f_i^t)- \mathbf{E}_\rho^t(S)}{t} \\
    &>  \frac{\sum_{i=1}^n\mathcal{E}_\rho^s(S',h_i^s\circ f_i^t)- \mathbf{E}_\rho^s(S)}{s} \\
    &> \sum_{i=1}^n\mathcal{E}_\rho(S',\pi\circ \tilde{f}_i^t)- \mathbf{E}_\rho(S).
    \end{align*}
    It follows from the minimizing property that  $$\frac{\mathbf{E}_\rho^t(S')- \mathbf{E}_\rho^t(S)}{t}> \frac{\mathbf{E}_\rho^s(S')- \mathbf{E}_\rho^s(S)}{s}$$ and  $$\frac{\mathbf{E}_\rho^t(S')- \mathbf{E}_\rho^t(S)}{t}>\mathbf{E}_\rho(S')- \mathbf{E}_\rho(S),$$ and hence the result follows.
\end{proof}

\end{subsection}
\end{section}

\begin{section}{Unstable equivariant minimal surfaces in $\mathbb{R}^n$}
We recall the setup of Theorem C. For $n \geq 2$ and $i=1,\dots, n$, let $\alpha_i$ be nonzero holomorphic $1$-forms on $S$ such that $\sum_{i=1}^n \alpha_i^2=0$. Let $\chi$ be the action of $\pi_1(S)$ on $\R^n$ corresponding to the 1-forms $\alpha_i$, and let $\rho$ be the action of $\pi_1(S)$ on a product $X = \prod_i (T_i, 2d_i)$ of trees corresponding to the quadratic differentials $\phi_i=\alpha_i^2$. We write $\mathcal{E}_\chi$ and $\mathcal{E}_\rho$ for the associated two-variable energies, and $\mathbf{E}_\chi$ and $\mathbf{E}_\rho$ for the energy functionals on Teichm{\"u}ller space. Let $h = (h_1, \ldots, h_n)$ and $\pi = (\pi_1, \ldots, \pi_n)$ be the $\chi$- and $\rho$- equivariant minimal maps respectively. 
\begin{subsection}{Isometric folding}
We begin with the statement (1) from Theorem C. The result is a consequence of the proposition below.
\begin{prop}\label{c1}
$\mathbf{E}_\rho\geq \mathbf{E}_\chi,$ with equality at $S$.
\end{prop}

The key is that there is a natural map $F: X \to \R^n$ intertwining $\rho$ and $\chi$. To see why, let's focus on a single tree $T_i$. Along a curve parametrizing a non-singular leaf for the vertical singular foliation of $\phi_i$, $\alpha_i$ evaluates the tangent vectors to purely imaginary numbers. Since $dh_i = \mathrm{Re}(\tilde{\alpha}_i)$, we deduce that $h_i$ is constant along the singular vertical foliation of $\tilde{\phi}_i$. Hence, $h_i$ descends to a map $F_i: T_i \to \R$, which we call the folding map of the tree. The map $F = (F_1, \ldots, F_n)$ has the required equivariance. 
 
\begin{lem}\label{isomf} If $S'$ is any point of $\mathbf{T}_g$, and $\pi_i'$ the unique $\rho_i$-equivariant harmonic map from $\tilde{S}'$ to $(T_i,2d_i)$, then the energy density of $\pi'_i$ is pointwise equal to the energy density of $F_i \circ \pi_i'$. 
\end{lem}
\begin{proof}
Let $\psi_i$ be the Hopf differential of $\pi_i'$. As discussed in Section 2, for any point $p$ at which $\psi_i(p)\neq 0$, there exists a neighbourhood $\Omega$ of $p$, an open interval $I\subset\mathbb{R}$, a map $\hat{\pi}_i':\Omega\to I$, and an isometric inclusion $\iota:I\to (T_i,2d_i)$ such that in $\Omega,$ $$\pi_i'=\iota\circ \hat{\pi}_i'.$$ By construction, the restriction of $F_i|_{\iota(I)}:\iota(I)\to \mathbb{R}$ is an isometric embedding. It follows by continuity that the energy densities are equal everywhere.
\end{proof}
%Set $F:\prod_{i=1}^nT_i\to\mathbb{R}^n$ be the product map $F=(F_1,\dots, F_n).$
\begin{proof}[Proof of Proposition \ref{c1}]
For any $S' \in \mathbf{T}_g$, let $\pi'$ be the $\rho$-equivariant harmonic map to the product of trees. The map $F\circ \pi'$ is a $\chi$-equivariant Lipschitz map to $\R^n$. By the minimizing property for harmonic maps, $$\mathbf{E}_\chi(S')\leq \mathcal{E}_\chi(S',F\circ \pi').$$
By the lemma above, $\mathcal{E}(S',F\circ \pi')=\mathbf{E}_\rho(S'),$ so we have $$\mathbf{E}_\rho \geq \mathbf{E}_\chi.$$ Working on the Riemann surface $S$, $h_i=F_i\circ \pi_i$ for every $i$, so $\mathbf{E}_\chi(S)= \mathcal{E}_\chi(S',F\circ \pi),$ and we have equality.
\end{proof}
\begin{remark}
Maps of the form $F_i\circ \pi_i'$ above are subtle. They are harmonic apart from some preimages under $\pi_i'$ of the vertices in $(T_i,2d_i)$, which are typically disjoint arcs or connected sums of disjoint arcs. Even though they have finite total energy, a Weyl lemma cannot be applied because they fail to be twice weakly differentiable on these lines. The map $x\mapsto |x|$ on $\mathbb{R}$  exhibits this type of behaviour.
\end{remark}

We see immediately from Proposition \ref{c1} that if $S$ is not a global (resp. local) minimum of $\mathbf{E}_\rho$, then it is not a global (resp. local) minimum of $\mathbf{E}_\chi$. So (1) is proved. Furthermore, we are very close to proving one direction of (2), once we recall the definition of the self-maps index, and its basic properties. We do this after collecting some standard facts about minimal surfaces in $\R^n$.
\end{subsection}

\begin{subsection}{Energy and area}
Let $f$ be any smooth $\chi$-equivariant map from $\tilde{\Sigma}_g$ to $\R^n$. The differential of $f$ descends to a closed $\R^n$-valued 1-form $\theta$ on $\Sigma_g$, and the cohomology class of $\theta$ is prescribed by the representation $\chi$. The map $f$ also defines a $\pi_1(\Sigma_g)$-invariant area form $dA_f = \sqrt{\mathrm{det}(\theta^T\theta)}$, and the area of $f$, which we write $A(f)$, is defined to be the integral of this form over $\Sigma_g$. If $S$ is a Riemann surface structure on $\Sigma_g$, then $\mathcal{E}_\chi(S,f) \geq A(f)$, with equality if and only if $f$ is minimal (in fact, the integrands are equal pointwise).

Now suppose we are in the setting of Theorem C, so that $h$ is a minimal $\chi$-equivariant map from $\tilde{S}$ to $\R^n$. Let $B$ be the branch locus of $h$ on $S$. 

\begin{lem}\label{atoe} Let $h_r$ be a smooth $\chi$-equivariant variation of $h$ such that $h_r = h$ in a neighborhood of $B$. Then for $r$ small enough, there is a smooth variation of Riemann surface structures $S_r$ such that $h_r$ is minimal with respect to $S_r$.
\end{lem}

\begin{proof}
For $r$ sufficiently small, the map $h_r$ is still an immersion away from $B$, and hence uniquely defines a new conformal structure on $S - B$. Since $X$ is compactly supported away from $B$, this conformal structure patches to the conformal structure of $S$ near $B$, and defines a new conformal structure $S_r$ on $S$, with respect to which $h$ is minimal.
\end{proof}

We say that a smooth $\R^n$-valued vector field $W$ on $S$ supported on $S-B$ is a normal variation of $h$ if it is perpendicular to the image of $dh$ at each point of $S - B$. For any such $W$, let $\tilde{W}$ be the pullback to $\tilde{S}$; then the family $h_r = h + r\tilde{W}$ is a $\chi$-equivariant deformation of $h$ equal to $h$ on a neighborhood of $B$. Taking the derivative of the corresponding $S_r$ at $r=0$ defines a linear map from the space of normal variations supported on $S-B$ to the tangent space of Teichm\"uller space at $S$. Let $V$ the graph of this map, viewed as a subspace of $T_S\mathbf{T}_g \times T_h \mathrm{Map}_\chi(\tilde{\Sigma}_g, \R^n)$. We have shown that restricted to $V$, the Hessian of $\mathcal{E}_\chi$ at the critical point $(S, h)$ is equal to the Hessian of $A$ at the critical point $h$. The latter has the following formula:

\begin{prop}[Theorem 32 in \cite{Law}] If $W$ is a normal variation supported in $S-B$, the second derivative of the area of any equivariant variation $h_r$ with derivative $W$ at $r=0$ is given by the quadratic form
\begin{equation}\label{2v}
   Q(W) = \int_S |(d W)^N|^2 - |\langle k, W \rangle|^2 
\end{equation}
where $(dW)^N$ is the component of $dW$ normal to the image of $dh$, $k$ is the vector-valued second fundamental form of $h(S)$, and the second term is interpreted as the square norm of the scalar-valued 2-tensor $\langle k, W \rangle$.
\end{prop}

\end{subsection}

\begin{subsection}{Lifting to $\mathbb{R}$-trees via self-maps}
In this section, we study energy and area in the context of the $\rho$-equivariant harmonic maps to products of $\mathbb{R}$-trees. Specifically, we relate $Q$ to the quadratic form $\mathbf{L}:T_S\mathbf{T}_g\times H^0(S;TS)^n\to\mathbb{R}$ defined in the introduction, which defines the self-maps index for $\mathbf{E}_\rho$. Let $H_c^0(S - B,TS)$ be the subspace of $H^0(S, TS)$ of smooth vector fields supported on $S-B$. 

 The key to the proof of the second part of Theorem C is the result below.
\begin{lem} \label{liftingtt} Suppose that $W$ is a normal variation of $S$ with support in $S-B$, and such that $Q(W)<0.$ Then there exists a harmonic Beltrami form $\mu$ on $S$ and vector fields $V_1,\dots, V_n\in H^0_c(S-B,TS)^n$ such that $$\mathbf{L}(\mu,V_1,\dots,V_n)<0.$$
\end{lem}
\begin{proof}
Denote the coefficients of $W$ by $W^i$. For each $i = 1, \ldots, n$, let $V_i$ be the vector field which vanishes on $B$ and is equal to $W^i \nabla x^i / |\nabla x_i|^2$ on $S-B$, where $\nabla x^i$ is the gradient of the coordinate function $x^i$ on $S$, which is nonvanishing on $S - B$. We point out that $V_i$ has compact support on $S - B$. 

Let $f_i^W: \R \times S \to S$ be flow of $V_i$, so that $f_i^W(r,\cdot)=f_i^r(\cdot)$. Then the family $H: \R \times \tilde{S} \to \R^n$ defined by $H_i(r, p) = h_i\circ f_i^r(p)$ has derivative $W$ at time zero. Moreover, the family $\Pi: \R \times \tilde{S} \to \prod_i (T_i,2d_i)$ defined by $\Pi_i(r,p) = \pi_i\circ f_i^r(p)$ satisfies $F_i \circ \Pi_i = H_i$, where $F_i$ is the folding map from $T_i$ to $\R$. Let $\pi_r$ be the map $(\pi_1\circ f_1^r,\dots, \pi_n\circ f^r).$ By Lemma \ref{atoe}, there exists a $C^\infty$ variation of conformal structures $r\mapsto S_r$ along which $\mathcal{E}_\rho(S_r,\pi_r)=\mathcal{E}_\chi(S_r,h_r)=A(h_r),$ and we set $\mu$ to be the Beltrami form in $T_S\mathbf{T}_g$ tangent to this path at time zero. If $Q(W)<0$, then taking the second variation of $r\mapsto\mathcal{E}_\rho(S_r,\pi_r)$ yields $\mathbf{L}(\mu,V_1,\dots, V_n)=Q(W)<0.$
\end{proof}

\end{subsection}

\begin{subsection}{Log cutoff}

In order to construct destabilizing variations for $Q$, it is helpful to do away with the condition that $W$ is supported on $S - B$. First, we need to say what it means for $W$ to be a normal variation over all of $S$. The map $S - B \to \C\mathbb{P}^{n-1}$, which sends $p$ to the (one-dimensional) image of $(\alpha_1, \ldots, \alpha_n)$ at $p$, extends holomorphically to all of $S$ by clearing denominators. Thus, the normal bundle also extends analytically to all of $S$. The quadratic form $Q$ is still finite for normal variations that are not necessarily supported on $S - B$.

For normal variations $W$, which are not necessarily supported on $S-B,$ we will need to show that one can replace them with variations that are supported on $S-B$ without changing the value of $Q$ too much. This is the log cut-off trick. If $r$ is the radial coordinate in $\mathbb{C}$ then the function $\log(r)/\log(\delta^{-1}) + 2$, defined between $r  = \delta$ and $r = \delta^2$, is equal to $1$ for $r= \delta$ and 0 for $r = \delta^2$ and has Dirichlet energy
\begin{equation} \label{eqn: cutoff}
\begin{split}
\frac{1}{2}\int_{\delta^2}^\delta \frac{2 \pi}{r\log(\delta^{-1})^2} = \frac{\pi}{\log(\delta^{-1})}.
\end{split}
\end{equation}
The point is this this tends to zero as $\delta$ goes to zero. A good picture is that $\log r$ is the height coordinate on a cylinder conformal to the punctured disk, so our function is an affine function of the height of the cylinder, and its derivative is small. The extension of this function by 0 and 1 is Lipschitz. For very minor reasons, it will be convenient to use a smooth cutoff function, so we let $l_\delta(r)$ be a perturbation of $\log(r)/\log(\delta^{-1}) + 2$ which extends smoothly by 0 and 1 and has Dirichlet energy no more than $2\pi/\log(\delta^{-1})$.

We use this model to define a cut-off function as follows. For each point $p_i$ of $B$, fix a holomorphic coordinate $z_i$ with $z_i(p_i) = 0$. Then, for any value of $\delta$ small enough that each $z_i$ is defined on the ball of radius $\delta$ around $p_i$ and these balls do not overlap, let $\eta_\delta$ be the function on $S$ defined by
\begin{itemize}
\item $\eta_\delta(p) = l_\delta(|z_i|)$ if $\delta^2 \leq |z_i(p)| \leq \delta$ for some $i$
\item $\eta_\delta(p) = 0$ if $|z_i(p)| \leq \delta^2$ for some $i$
\item $\eta_\delta(p) = 1$ otherwise.
\end{itemize}

We now use the log cut-off trick to prove the following lemma.

\begin{lem}\label{cftrick} Suppose that $W$ is normal variation of $h$ on $S$. Then given any $\epsilon > 0$, there is a constant $d(\epsilon, Q(W), \sup |W|)$ such that for all $\delta < d$,
\begin{equation}\label{lcf}
    |Q(\eta_\delta W) - Q(W)| < \epsilon.
\end{equation}
\end{lem}

\begin{proof}
For $\delta$ to be determined, we compute $Q(\eta_\delta W)$. We first treat the normal term in the formula (\ref{2v}) applied to the variation $\eta_\delta W$:
\begin{align*}
    \int_\Sigma |(d(\eta_\delta W))^N|^2 &= \int_\Sigma|\eta_\delta(d W)^N + (W d\eta_\delta)^N|^2 \\
    &= \int_{\delta^2\leq |z|\leq \delta} |\eta_\delta(d W)^N + (Wd\eta_\delta)^N|^2+\int_{|z|\geq \delta}|(d W)^N |^2,
\end{align*}
where $W d\eta_\delta$ is the $\mathbb{R}^n$-valued $1$-form $W\otimes d\eta_\delta.$ Hence,
\begin{align*}
    |Q(\eta_\delta W) - Q(W)|  &\leq \int_{\delta^2\leq |z|\leq \delta} |\eta_\delta(d W)^N + (Wd\eta_\delta)^N|^2+\int_{|z|\leq \delta}|(d W)^N |^2+\int_S (1-\eta_\delta^2)|W|^2|k|^2  \\
    &= \int_{\delta^2\leq |z|\leq \delta} |\eta_\delta(d W)^N + (Wd\eta_\delta)^N|^2+ O(\delta^2),
\end{align*}
since $1-\eta_\delta^2$ is supported in $|z|\leq \delta.$ By Cauchy-Schwarz and (\ref{eqn: cutoff}),
\begin{align*}
 |Q(\eta_\delta W) - Q(W)|  &\leq  \int_{\delta^2\leq |z|\leq \delta} |\eta_\delta(d W)^N|^2 + |(Wd\eta_\delta)^N|^2 + 2|\eta_\delta(d W)^N|^2|(W d\eta_\delta)^N|^2 +O(\delta^2)\\
    &= O(\delta^2)+ O\Big (\frac{1}{\log \delta^{-1}}\Big )=O\Big (\frac{1}{\log \delta^{-1}}\Big ).
\end{align*}
Thus, we can choose $\delta>0$ so that the difference of second variations is at most $\epsilon.$
\end{proof}
\end{subsection}

A consequence is that we can speak without ambiguity of the index of $Q$.

\begin{prop}\label{sminusb} The index of $Q$ on the space of all normal variations is equal to the index of $Q$ on the subspace of normal variations supported in $S - B$.
\end{prop}

\begin{proof} We just need to show that if there is a $k$-dimensional space of normal variations on which $Q$ is negative definite, then there is another $k$-dimensional space of normal variations supported in $S-B$ on which $Q$ is still negative definite. Let $V$ be a $k$-dimensional space of normal variations on which $Q$ is negative definite. Let $S(V)$ be the unit sphere in $V$ with respect to any metric on $V$. Then for $\delta$ small enough, $Q(\eta_\delta W) < 0$ for every $W \in S(V)$. Since this implies $\eta_\delta W \neq 0$, the space $\{\eta_\delta W | W \in V\}$ is a $k$-dimensional subspace of normal variations supported in $S-B$ on which $Q$ is negative definite.
\end{proof}

We may now finish the proof of Theorem C.

\begin{proof}[Proof of Theorem C (2)]

Let $k$ be the index of $\mathbf{E}_\chi$, and let $W \subset T_S \mathbf{T}_g$ be a $k$-dimensional subspace on which the second variation is negative definite. By the implicit function theorem, the unique harmonic $1$-form in a given cohomology class varies smoothly with the conformal structure of $S$. We can integrate this smoothly-varying $1$-form to give a smooth equivariant variation of the harmonic map $h$. Projecting the variation onto the normal bundle, we get from $W$ a vector space of normal variations of $h$ on which the second derivative of $\mathbf{E}_\chi$ is equal to $Q$. Since it is assumed to be positive definite, this space is still $k$ dimensional. 

By Proposition \ref{sminusb}, we can replace this with a $k$-dimensional subspace of normal variations supported on $S-B$ on which $Q$ is still negative definite. Then by Lemma \ref{liftingtt}, there is a $k$-dimensional subspace of $T_S\mathbf{T}_g \times H_c^0(S - B, TS)^n$ on which $\mathbf{L}$ is negative definite, and so the index of $\mathbf{E}_\rho$ by self-maps is at least $k$.

In the other direction, suppose $W'$ is a $k$-dimensional subspace of $T_S\mathbf{T}_g \times H^0(S,TS)^n$ on which $\mathbf{L}$ is negative definite. Since $\mathbf{L}$ is positive semidefinite on $\{0\} \times H^0(S,TS)^n$, the projection of $W'$ to $T_S\mathbf{T}_g$ is still $k$-dimensional. For maps to manifolds, the positive semidefinite property follows from the computation \cite[Theorem H]{Har}, and we get the same result in our setting by repeating the computation but using the measurable energy density with the characterization (\ref{mease}). Since $\mathbf{E}_\rho$ is an infimum over all maps, we get an upper bound for $\mathbf{E}_\rho$ near $S$ by a smooth function with negative definite Hessian at $S$. Recall that Proposition \ref{c1} says that $\mathbf{E}_\chi \leq \mathbf{E}_\rho$, and so the index of $\mathbf{E}_\chi$ at $S$ is at least $k$.
\end{proof}

\end{section}

\begin{section}{The general case}
In this section, we generalize Theorem C to the situation in which the quadratic differentials are not necessarily squares of abelian differentials. We then specialize to dimension $3$ and give the proof of Theorem D.
\begin{subsection}{The spectral curve}
Let $S_0$ be a point of $\mathbf{T}_g$, and let $\phi_1, \ldots, \phi_n$ be nonzero holomorphic quadratic differentials on $S_0$ summing to zero. To this data, there is an associated spectral curve. By this, we mean a branched covering $S$ of $S_0$ and abelian differentials $\alpha_i$ on $S$ that square to the pullback of $\phi_i$, which is terminal in the sense that if the $\alpha_i$ lift to squares on some other branched cover $R\to S_0$, then this factors through a branched covering from $R\to S$. It is always a $2^n$-fold branched covering of $S_0$, but may be disconnected, for instance if any $\phi_i$ is already a square. By universality, $S$ has $n$ holomorphic involutions $\tau_i$, each of which negates $\alpha_i$ and fixes $\alpha_j$ for $j \neq i$.

We let $\rho$ be the action of $\pi_1(S_0)$ on the product $X$ of the $\R$-trees $(T_i,2d_i)$ corresponding to the quadratic differentials $\phi_i$, and $\pi$ the canonical equivariant map from $\tilde{S}_0$ to $X$. 

Since $S$ has $n$ abelian differentials whose squares sum to zero, the theory of the previous section applies. For instance, we can integrate $\mathrm{Re}(\tilde{\alpha}_i)$ on a simply connected covering space to get a harmonic map $h$ to $\R^n$, equivariant under a representation $\chi$ of the Deck group, and well defined up to a constant on each component of $S$. The energy density of this map descends not only to $S$, but all the way to $S_0$, where it is equal to the energy density of $\pi$.

In the spirit of Proposition \ref{sminusb}, we want to compare the index of $\mathbf{E}_\rho$ through self-maps at $S_0$ to the index of the quadratic form $Q$ associated to $h$. But to get the right comparison, we need to restrict $Q$ to a subspace of the space of normal variations. Let $G \cong (\Z/2\Z)^n$ be the group generated by the $\tau_i$. Let $\sigma$ be the action of $G$ on $\R^n$ such that each $\tau_i$ acts by reflection in the $i$th coordinate hyperplane. Let $NV^\sigma$ be the space of normal variations of $h$ that are $\sigma$-equivariant. 

\begin{prop} The index of $\mathbf{E}_\rho$ by self-maps is equal to the index of $Q$ on $NV^\sigma$.
\end{prop}

\begin{proof}
Let $k$ be the index of $Q$ on $NV^\sigma$. The first thing we want to do is use Proposition \ref{sminusb} to find a $k$-dimensional space of $\sigma$-equivariant normal variations on $S-B$ on which $Q$ is still negative definite. This works fine if we choose our cutoff function $\eta_\delta$ to be $\tau_i$-invariant. For instance, we can define $\eta_\delta$ to be the pull-back to $S$ of the similarly-defined function on $S_0$; then the dependence of the energy of $\eta_\delta$ with $\delta$ is the same up to a factor of 2 coming from the relation $\log(|\sqrt{z}|) = \log(|z|)/2$.

Next, for every $W$ in this space, we get $n$ tangential vector fields $V_i = W^i \nabla x^i / |\nabla x^i|^2$ on $S$, as in Proposition \ref{liftingtt}. Since both $W^i$ and $dx^i$ transform the same way under each $\tau_j$, we have $\tau_j(V_i) = \pm W^i (\pm \nabla x^i) / |\nabla x^i|^2$, where each sign is $+$ if $i \neq j$ and $-$ if $i=j$. Hence each $V_i$ descends to a vector field on $S_0 - B$, which we still call $V_i$. 

For each $i$, let $f_i^W:\mathbb{R}\times S\to S$, $f_i^W(r,\cdot)=f_i^r(\cdot)$ be the flow of $V_i$. Let $h_i$ be the component functions of $h$, $h_i^r=h_i\circ f_i^r$, and $h_r=(h_1^r,\dots, h_n^r).$ The conformal structures $S_r$ for which each $h_r$ is conformal are still $G$-invariant, hence descend to conformal structures $(S_0)_r$ on $\Sigma_g$. Let $\pi^r_i = \pi \circ f_i^r$. Even though the tree $T_i$ no longer folds to $\R$, the energy density of $\pi^r_i$ on $(S_0)_r$ is still pointwise equal to the energy density of $h^r_i$ on $S^t$; indeed, both are equal to $|(f_i^r)^*\mathrm{Re}(\alpha_i)|^2$. Therefore, the second derivative of $\mathcal{E}_\rho$ is equal to $Q$ on this $k$-dimensional space so the index of $\mathbf{E}_\rho$ by self-maps is at least $k$.

The other inequality is easier. If the index of $\mathbf{E}_\rho$ by self-maps is $k$, then we can use the log-cutoff trick to find a $k$-dimensional space of vector fields $V_i$ supported on $S_0-B$ and variations $\mu_i$ of conformal structure on which $L$ is negative definite. Lifting everything to $S$ and differentiating the coordinate functions, we get a $k$-dimensional space of equivariant variations of $h$ for which the second derivative of energy is negative definite. Taking the normal components of these variations, and using that energy dominates area, we get a $k$-dimensional subspace of $NV^\sigma$ on which $Q$ is negative definite.
\end{proof}

\end{subsection}

\begin{subsection}{Unstable minimal surfaces in $\R^n$}

In order to finish the proof of Theorem A, we need to construct for each $g \geq 2$ and $n \geq 3$, either an unstable equivariant minimal surface $S$ of genus $g$ in $\R^n$, or a surface $S_0$ of genus $g$ whose spectral curve is a $(\Z/2\Z)^n$-equivariantly unstable minimal surface in $\R^n$.

If $g \geq 3$, then as we discuss in the next section, there are plenty of equivariant minimal surfaces of genus $g$ in $\R^n$. They are not always unstable; for instance, if the minimal map is holomorphic with respect to some complex structure on a linear subspace of $\R^n$, then it is calibrated by the K{\"a}hler form, and hence stable. In general, it is not straightforward to decide if a minimal surface in a flat space is unstable.

A special case is when the equivariant minimal surface is contained in a real 2-plane, and hence is stable. We call such a minimal surface flat. These at least are easy to identify. 

\begin{prop} \label{k=0} Let $\phi_1,\dots,\phi_n\in QD(S_0)$ sum to $0$, giving a $\chi$-equivariant map $h:\tilde{S}\to\mathbb{R}^n$ as before. The vector valued second fundamental form $k$ of $h(\tilde{S})$ vanishes identically if and only if the quadratic differentials $\phi_i$ are all complex multiples of one another.
\end{prop}
\begin{proof} Let $h_1,\dots, h_n$ denote the coordinate functions of $h$. Since $\phi_i = ((h_i)_z)^2dz^2$, the quadratic differentials are all complex multiples of one another if and only if the functions $(h_i)_z$ are. In one direction, assume $(h_i)_z = a_i f(z)$ for some function $f(z)$ and some complex constants $a_i$. Then the image of the $\mathbb{R}^n$-valued $1$-form $dh$ is contained in a two-dimensional subspace, and by integrating, we see that image of $h$ is contained in an affine subspace of $\mathbb{R}^n$. In particular, it is totally geodesic, so the second fundamental form is zero. Conversely, if the second fundamental form is zero, then the image of $dh$ is contained in some two-dimensional linear subspace, and so the image of $h_z$ is contained in the complexification of that subspace, which is two-dimensional. As $h$ is weakly conformal, $\langle h_z, h_z \rangle = 0$; since the inner product is nondegenerate on the complexification of any real two-dimensional subspace, this shows that $h_z$ is contained in a complex line (we use analyticity to deduce this as well at the branch points), and so the functions $(h_i)_z$ are all complex multiples of one another.
\end{proof}

For the remaining section, we restrict to $n=3$. For $n\geq 3$, any isometric inclusion of $\R^3$ into $\R^n$ gives examples in $\mathbb{R}^n$. Let $\mathbf{M}_g$ be the moduli space of Riemann surfaces of genus $g$, and let $E^n$ be the total space of the bundle over $\mathbf{M}_g$ consisting of $n$-tuples of quadratic differentials that sum to $0$. Instability of the corresponding equivariant minimal surfaces in $\mathbb{R}^n$ is an open condition on $E^n$, so by perturbing the $3$-dimensional examples we get many more.

\end{subsection}

\subsection{Equivariant minimal surfaces in $\R^3$}\label{5.3}
Every non-flat equivariant minimal surface in $\R^3$ is unstable. Indeed, in dimension 3, the expression $|\langle k, W \rangle|^2$ in the formula for $Q(W)$ is equal to $2|K||W|^2$, where $K$ is the Gauss curvature of the equivariant minimal surface. The normal bundle to the minimal surface $h(\tilde{S})$ is a real line bundle on $S$. Since $S$ is always orientable, the normal bundle is as well, and hence it is equivariantly trivial. If $N$ is a unit normal section, and $\eta$ is any function on $S$, then the second variation formula (\ref{2v}) takes the form $$Q_0(\eta N)=\int_\Sigma |\nabla \eta|^2 - |K|^2|\eta|^2.$$ As long as the curvature $K$ is anywhere nonzero, a constant section of $N$ will therefore be destabilizing: for $\eta=1,$ 
\begin{equation*}
    Q_0(N)=\int_\Sigma - |K|^2<0.
\end{equation*}

When $g \geq 3$, the moduli space of $(S, \alpha_1, \alpha_2, \alpha_3)$, where $S$ is a Riemann surface of genus $g$ and $\alpha_i$ are abelian differentials on $S$ whose squares sum to zero, but are not all mutliples of one another, is nonempty and has complex dimension $3g$ (in \cite[Section 6]{Petri} it is shown that the quotient by the natural free actions of $\mathbb{C}^*$ and $\textrm{SO}(3,\mathbb{C})$ has dimension $3g-4$).  This proves Theorem A for $g \geq 3$.

\begin{remark}
In fact, in \cite[Theorem 16]{Ros}, Ros proves that every non-flat minimally immersed surface of genus $g$ in a 3-torus has index at least $2g/3 - 1$. The result easily generalizes to any non-flat equivariant minimal immersion for any representation, but we emphasize that it applies only to immersed surfaces. 
\end{remark}

Unfortunately, there are no non-flat equivariant minimal surfaces of genus 2 in $\R^3$, stable or not. This is because the canonical map lands in $\mathbb{P}^1$, so the canonical curve cannot be contained in a rank 3 quadric (or see the comment after Proposition \ref{g=2}). Hence, we are forced to study $\sigma$-equivariant deformations of the spectral curve. The key that makes this work is that the normal bundle of the spectral curve $S$ of $(S_0, \phi_1, \phi_2, \phi_3)$ can be equivariantly trivial even if the $\phi_i$ are not squares (in which case $S$ is just $8$ copies of $S_0$).

\begin{prop} Suppose that the sextic differential $\phi_1\phi_2\phi_3$ is the square of a cubic differential $c$. Then there is a $\sigma$-equivariant deformation of $S$ of constant length 1.
\end{prop}

\begin{proof}
The cubic differential $c$ distinguishes two components of $S$; one on which $\alpha_1\alpha_2\alpha_3 = c$, and one on which it is equal to $-c$. Each $\tau_i$ interchanges the two components of $S$. The subgroup $\Gamma<(\Z/2\Z)^3$ preserving the components acts on $\R^3$ in an orientation-preserving way. Indeed, for each element $\gamma\in \Gamma$, the determinant of the matrix describing the product of hyperplane reflections is equal to the product of the monodromies of the $\alpha_i$ under the action of $\gamma$. We can use the orientation of $\R^3$, together with the orientation of the component of $S$, to equivariantly orient the normal bundle. Since the normal bundle is a line bundle, it therefore has an equivariant section of constant length. 
\end{proof}

\begin{remark} \label{abeliancubic}
If each $\phi_i$ is the square of an abelian differential $\alpha_i$, then clearly $\phi_1\phi_2\phi_3 = c^2$ with $c = \alpha_1\alpha_2\alpha_3$. 
\end{remark}

If the quadratic differentials $\phi_i$ are not complex multiples of one another, then neither are their lifts $\alpha_i$ to the spectral curve. Hence, the minimal map from the lift of the spectral curve is non-flat, so any $\sigma$-equivariant deformation of constant length will be destabilizing.

The final step is to show that there are non-flat solutions even in genus 2 to the equations $\phi_1\phi_2\phi_3 = c$ and $\phi_1 + \phi_2 + \phi_3 = 0$. For any $g\geq 2$, let $\mathcal{P}_g$ be the moduli space of genus $g$ Riemann surfaces $S$ together with a triple of quadratic differentials $\phi_i$ summing to zero whose product is a square and which are not all complex multiples of one another.

\begin{prop} \label{g=2} The moduli space $\mathcal{P}_2$ has dimension $3$.
\end{prop}

\begin{proof}
Consider the three dimensional family of algebraic curves $w^2 = z(z-1)(z-a)(z-b)(z-c)$ for $(0,1,a, b, c)$ distinct complex numbers. This is a finite covering of the moduli space of genus 2 Riemann surfaces. Every holomorphic quadratic differential on a curve in this family is of the form $p(z) (dz)^2/w^2$ for $p(z)$ a polynomial of degree at most 2. If the roots of $p(z)$ are branch points of the curve, then the quadratic differential vanishes to order two at the corresponding point of the curve. For arbitrary $a$ and $b$, and $c$ to be determined, let
\begin{align*}
\phi_1 &= z(z-1)\frac{dz^2}{w^2} \\ 
\phi_2 &= \mu(a,b) (z-a) \frac{dz^2}{w^2},
\end{align*}
where $\mu(a,b) = - b(b-1)/(b-a)$ is chosen so that $\phi_1 + \phi_2$ vanishes at $b$ (equivalently, that the corresponding quadratic polynomial vanishes at $b$). A short computation shows that the other root of the polynomial for $\phi_1 + \phi_2$ is at $a(b-1)/(b-a)$, so if this happens to be the value of $c$, then the sextic differential $\phi_1\phi_2\phi_3$ vanishes to order two at each of the six branch points of the curve (including $\infty$). Hence it is the square of the cubic differential $dz^3/w^2$, which vanishes to order one at each of these points. Including a parameter rescaling $\phi_1$, $\phi_2$, and $\phi_3$, this shows that $\mathcal{P}_2$ has dimension 3.
\end{proof}

For example we could take $a = -1, b = i,$ and $c = -i$ to get a solution on the hyperelliptic curve $w^2 = z^5-z$. This suffices for the proof of Theorem A.

\begin{remark} Note that the triples of quadratic differentials in genus 2 are squares of abelian differentials since the polynomials $z(z-1)$, etc., are not squares. However, they still have even order zeros.
\end{remark}

Together with \cite[Section 6]{Petri} and the Remark \ref{abeliancubic}, this shows: 

\begin{prop}\label{3g-3}
For every genus $g \geq 2$, $\mathcal{P}_g$ is nonempty and every component has complex dimension at least $3g-3$.
\end{prop}

We give a self-contained proof of this proposition, since it is very brief in the reference.

\begin{proof}
We have already proved this for genus 2. The canonical map of a hyperelliptic curve of genus 3 is the vanishing locus of a nondegenerate quadric on $\mathbb{CP}^2$; diagonalizing this quadric gives three abelian differentials whose squares sum to zero on the curve. By Remark \ref{abeliancubic}, these give points in $\mathcal{P}_3$. Since the hyperelliptic locus has dimension 5, we get a sixth dimension from rescaling the abelian differentials. This proves the result for $g=3$.

In general, taking unramified coverings of a point in $\mathcal{P}_2$ shows that $\mathcal{P}_g$ is nonempty for every $g$. To get the bound on dimension, we observe that $\mathcal{P}_g$ is, up to a double cover, the intersection in the total space of the bundle $H^0(K^3)$ over $\mathbf{M}_g$ (dimension $14(g-1)$) of the sextic differentials that are squares of cubic differenitals (dimension $8(g-1)$) and those that are the product of three independent quadratic differentials summing to zero (dimension $9(g-1)$). This gives a lower bound on the dimension of $8(g-1) + 9(g-1) - 14(g-1) = 3(g-1)$. 
\end{proof}

\end{section}

\bibliographystyle{plain}
\bibliography{bibliography}

\end{document}